\documentclass[11pt,a4paper]{article}
\usepackage{mycommandsNew}

\setgcgap{0}
\setgcscale{0.32}

\newcommand{\vpermuton}{{\setgcscale{0.16}\setgcextra{\draw[very thin](0,0)rectangle(2,2);}\gcone{0}{-2,2}}}

\title{\textbf{Independence of permutation limits \\at infinitely many scales}}

\author{$\phantom{{}^\dagger}$David Bevan${}^\dagger$}

\hypersetup{
pdftitle={Independence of permutation limits at infinitely many scales},
pdfauthor={David Bevan},
pdfstartview={XYZ null null 1.00}
}

\date{}

\begin{document}
\maketitle

{\begin{NoHyper}
\let\thefootnote\relax\footnotetext
{${}^\dagger$Department of Mathematics and Statistics, The University of Strathclyde, Glasgow, Scotland.}
\end{NoHyper}}

{\begin{NoHyper}
\let\thefootnote\relax
\end{NoHyper}}

\begin{abstract}
\noindent
We introduce a new natural notion of convergence for permutations at any specified scale, in terms of the density of patterns of restricted width.
In this setting we prove that 
limits may be chosen independently at a countably infinite number of scales.
\HIDE{
Firstly, we exhibit a sequence of permutations $(\zeta_j)$ such that, for each irreducible $p/q\in\bbQ\+\cap\+(0,1]$,
a 
subpermutation of $\zeta_j$ of width at most $|\zeta_j|^{p/q}$ is
a.a.s. increasing if $q$ is odd, and is
a.a.s. decreasing if $q$ is even.
In the second, we construct a sequence of permutations $(\eta_j)$ such that,
for every skinny monotone grid class $\Grid(\mathbf{v})$, there is a function $f_\mathbf{v}$ such that
any 
subpermutation of $\eta_j$ of width at most $f_\mathbf{v}(|\eta_j|)$ is a.a.s. in $\Grid(\mathbf{v})$.
}
\end{abstract}

\section{Introduction}

We study pattern densities in permutations.
Let $S_n$ denote the set of permutations of length~$n$.
An \emph{occurrence} of pattern $\pi\in S_k$ in permutation $\sigma\in S_n$ (with $k\leqs n$)
is a $k$-element subset of indices $1\leqs i_1\leqs \ldots \leqs i_k\leqs n$ whose image $\sigma(i_1)\ldots\sigma(i_k)$
under $\sigma$ is order-isomorphic to~$\pi$.
If $\pi$ occurs in $\sigma$, then $\pi$ is a \emph{subpermutation} of~$\sigma$.
For example, $132$ is a subpermutation of $35142$, since $35142$ contains two occurrences of the pattern $132$.

Let $\nu(\pi,\sigma)$ be the number occurrences of $\pi$ in~$\sigma$.
Then the \emph{global density} of $\pi$ in $\sigma$, which we denote $\rho(\pi,\sigma)$, is $\nu(\pi,\sigma)\big/\binom{n}{k}$.
Observe that
$
\rho(\pi,\sigma) = \prob{\sigma(K) = \pi}
$,
where $K$ is drawn uniformly from the $k$-element subsets of~$[n]$, and $\sigma(K)$ denotes the permutation order-isomorphic to the image of $K$ under~$\sigma$.

We say that an occurrence $i_1\leqs\ldots\leqs i_k$ of $\pi$ in $\sigma$, has \emph{width} $i_k-i_1+1$.
Given a real number $f\in[k,n]$, let $\nu_f(\pi,\sigma)$ be the number of occurrences having width no greater than~$f$.
Then the density of $\pi$ in $\sigma$ \emph{at scale $f$}, denoted $\rho_f(\pi,\sigma)$, is $\nu_f(\pi,\sigma)\big/\binom{n}{k}_{f}$,
where
\[
\binom{n}{k}_{\!f} \;=\; \sum_{w=k}^{\floor{f}}(n - w + 1)\binom{w-2}{k-2}
\]
is the number of $k$-element subsets of $[n]$ of width at most~$f$. Thus,
$
\rho_f(\pi,\sigma) = \prob{\sigma(K) = \pi}
$,
where $K$ is drawn uniformly from the $k$-element subsets of $[n]$ of no greater than~$f$.
Clearly, $\rho_{n}(\pi,\sigma)$ is the same as~$\rho(\pi,\sigma)$.
And $\rho_{k}(\pi,\sigma)= \nu_{k}(\pi,\sigma) / (n - k + 1)$ is the density of \emph{consecutive} occurrences of the pattern $\pi$ in~$\sigma$; namely, the \emph{local density} of $\pi$ in~$\sigma$.

The scale $f$ can be envisaged as specifying a ``zoom level'' or ``magnification'': the horizontal extent
of a window through which we inspect a permutation.
We typically consider the scale $f$ to be a function $f(n)$ of the length, $n$, of the host permutation~$\sigma$, such as ${\log n}$, ${\sqrt{n}}$ or ${n/\log n}$.
With a slight abuse of notation, we omit the argument when it is clear from the context.

We say that a function $f:\bbN\to\bbR^+$ is a \emph{scaling function} if $1\ll f(n)\ll n$ and $f(n)\leqs n$ for all $n$, where
we write $f(n)\ll g(n)$ 
to denote that $\liminftyt f(n)/g(n) = 0$.
Note that we require scaling functions to tend to infinity and also to be sublinear.
Our interest is in the behaviour of pattern density at different scales as $n$ tends to infinity.

Two recent papers have investigated the density of patterns at different scales.
In~\cite{BevanLocallyUniform}, the following scenario is considered.
Suppose $\sigma$ is drawn uniformly from those permutations in $S_n$ containing exactly $m$ \emph{inversions} ($21$ patterns); that is,
$\nu(21,\sigma) = m$. Moreover, suppose that $m=m(n)$ satisfies $n\ll m\ll n^2/\log^2 n$.
Then clearly $\rho(21,\sigma)\ll 1/\log^2 n \to0$, and indeed it is shown that $\rho_f(21,\sigma)\to0$ as long as $f\gg m/n$.
However, at smaller scales, two points are as likely to form an inversion as not: if $f\ll m/n$, then $\rho_f(21,\sigma)\to\nhalf$.
Thus, the local structure of $\sigma$ reveals nothing about its global form.

Borga and Penaguiao~\cite{BP2020} consider the general relationship between asymptotic global pattern densities and asymptotic local pattern densities,
and prove that they are \emph{independent} in the following sense:
Given a set of patterns $G$
and any consistent combination of their asymptotic global pattern densities 
$\Gamma\in[0,1]^G$,
and similarly a set of consecutive patterns $L$
and any consistent combination of their asymptotic local pattern densities 
$\Lambda\in[0,1]^L$,
then there exists a sequence of permutations $(\sigma_j)_{j\in\bbN}$ such that
$\rho(\pi,\sigma_j)\to\Gamma_\pi$ for each $\pi\in G$,
and $\rho_{|\tau|}(\tau,\sigma_j)\to\Lambda_\tau$ for each $\tau\in L$.

Our main result is that
this independence can be extended to infinitely many scales.
If any consistent combination of asymptotic pattern densities is chosen for each of a countably infinite number of suitably distinct scales, then
there exists a sequence of permutations for which all the limiting densities at each scale match the choices.

In the next section we look at various notions of convergence for permutations.
To begin, we recall the basic results concerning the global convergence of a sequence of permutations. 
We then introduce and investigate an approach to defining convergence at a specified scale, and also present a stricter notion of convergence in which the choice of scale is irrelevant.
Finally, we briefly recall the essential results concerning local convergence.

In Section~\ref{sectScaleSpecificLimits}, we prove a number of results concerning convergence at a given scale, most notably (Theorem~\ref{thmLimScaleAny}) that
scale limits can be limits at any scale.
Then, in Section~ \ref{sectIndependence}, building on these results, we prove our main theorem (Theorem~\ref{thmIndependence}) showing asymptotic independence at a countably infinite number of scales. After briefly presenting two example constructions, we extend this result to two dimensions (Theorem~\ref{thm2DIndependence}).

\section{Notions of convergence}\label{sectNotionsOfConvergence}

\subsection{Global convergence}\label{sectGlobalConv}

An infinite sequence $(\sigma_j)_{j\in\bbN}$ of permutations with $|\sigma_j|\to\infty$ is \emph{globally convergent} if $\rho(\pi,\sigma_j)$ converges for every permutation $\pi$.
To every convergent sequence of permutations one can associate an analytic limit
object. A \emph{permuton} is a probability measure $\Gamma$ on the $\sigma$-algebra of Borel sets of the unit square $[0,1]^2$ such that $\Gamma$ has \emph{uniform marginals}. That is, for every interval $[a,b]\subseteq[0,1]$, we have $\Gamma\big([a,b]\times[0,1]\big)=\Gamma\big([0,1]\times[a,b]\big)=b-a$.

Given a permuton $\Gamma$ and an integer $k$, we can randomly sample $k$ points $(x_1,y_1),\ldots,(x_k,y_k)$ in $[0,1]^2$ from the measure $\Gamma$.
With probability one, their $x$- and $y$-coordinates are distinct, because $\Gamma$ has uniform marginals.
So, from these points, we can define a permutation $\pi$ as follows. If we list the $x$-coordinates in increasing order $x_{i_1}<\ldots<x_{i_k}$, then $\pi$ is the unique permutation order-isomorphic to $y_{i_1}y_{i_2}\ldots y_{i_k}$.
We say that a permutation sampled in this way from a permuton $\Gamma$ is a \emph{$\Gamma$-random permutation} of length $k$.

This sampling approach is used to define a notion of pattern density for permutons.
If $\Gamma$ is a permuton and $\pi$ is a permutation of length $k$, then $\rho(\pi,\Gamma)$ is
the probability that a $\Gamma$-random permutation of length $k$ equals $\pi$.

We now recall the core results from~\cite{HKMRS2013,HKMS2011b}.
For every convergent sequence $(\sigma_j)_{j\in\bbN}$ of
permutations, there exists a unique permuton $\Gamma$ such that
\[
\rho(\pi,\Gamma) \;=\; \liminfty[j] \rho(\pi,\sigma_j) \text{~~for every permutation $\pi$} .
\]
This permuton is the \emph{limit} of the sequence $(\sigma_j)_{j\in\bbN}$.

Conversely, if $\Gamma$ is a permuton and, for each~$j\in\bbN$, $\sigma_j$ is a $\Gamma$-random permutation of length $j$, then with
probability one the sequence $(\sigma_j)_{j\in\bbN}$ is convergent, and $\Gamma$ is its limit.
We call such a sequence (that converges to $\Gamma$) a \emph{$\Gamma$-random sequence}.

Permutons were introduced in \cite{HKMRS2013,HKMS2011b,HKMS2011} employing a different but equivalent definition.
The measure theoretic view presented above was originally used in~\cite{PS2010},
and was later exploited in \cite{GGKK2015}, in which the term ``permuton'' was first used.
Subsequently, among other things, permutons have been applied
to the investigation of the {feasible region} of possible pattern densities~\cite{GHKKKL2017,KKRW2020},
to the characterisation of {quasirandom} permutations~\cite{KP2013,CKNPSV2020},
and to determining the shapes of permutations in substitution-closed permutation classes~\cite{BBFGP2018,BBFGMP2020} (introducing the notion of the random Brownian separable permuton~\cite{Maazoun2020}).

\subsection{Convergence at specified scales}

In an analogous manner to the definition of global convergence, we introduce a notion of convergence at a specified scale.
Given a scaling function $f$, an infinite sequence $(\sigma_j)_{j\in\bbN}$ of permutations with $|\sigma_j|\to\infty$ is \emph{convergent at scale $f$} if $\rho_{f}(\pi,\sigma_j)$ converges for every permutation~$\pi$.

If $(\sigma_j)_{j\in\bbN}$ is convergent at scale $f$, then there exists an infinite vector $\Xi\in[0,1]^S$ (where $S$ is the set of all permutations) such that $\rho_{f}(\pi,\sigma_j)\to\Xi_\pi$ for all $\pi\in S$.
Note that, for any $k\geqs1$, we have $\sum_{\pi\in S_k}\Xi_\pi=1$.
In the current context, we consider $\Xi$ itself to be the {limit} of the sequence at scale $f$.
We call these limits \emph{scale limits}.
We
prove below (Theorem~\ref{thmLimScaleAny}) that scale limits can be limits at any scale in the following sense:
If $\Xi$ is any scale limit and $f$ any scaling function, then there exists a sequence of permutations which converges at scale $f$ to~$\Xi$.

Sometimes there exists a (unique) permuton $\Gamma$ such that $\rho(\pi,\Gamma)=\Xi_\pi$ for every $\pi$.
If such a $\Gamma$ does exist, then we interchangeably deem either $\Xi$ or $\Gamma$ to be the limit.
However, in general this is not the case.

Suppose $\Gamma_\textsf{V}$ is the \textsf{V}-shaped permuton 
$\!\vpermuton\!$,
in which the mass is uniformly distributed along the decreasing diagonal of the left half
and uniformly distributed along the increasing diagonal of the right half.
Suppose $(\sigma_j)_{j\in\bbN}$ is a $\Gamma_\textsf{V}$-random sequence and $f$ is a scaling function.
If~$K$ is drawn uniformly from subsets of $[j]$ of width no greater than $f(j)$, then the probability that $\sigma_j(K)$ is not monotone is no greater than $(f(j)-2)/(j+1-f(j))\sim f(j)/j$, which tends to zero with increasing~$j$.
So, $\rho_f(\pi,\sigma_j)\to\nhalf$ if $\pi$ is an increasing or decreasing pattern (by the law of large numbers), and $\rho_f(\pi,\sigma_j)\to0$ for all other~$\pi$.

Note that this limit is not equal to \emph{any} permuton.
Specifically, there is no permuton $\Gamma$ such that $\rho(12,\Gamma)=\rho(21,\Gamma)=\nhalf$ but \mbox{$\rho(\pi,\Gamma)=0$} for all non-monotone $\pi\in S_3$.
Indeed, it seems natural to consider the limit of a \mbox{$\Gamma_\textsf{V}$-random} sequence at scale $f$ to be, in some sense, equal to $\half\!\gcone{1}{1}\!+\half\!\gcone{1}{-1}\!$.
In general, we believe that certain probability distributions over permutons (that is, certain \emph{random permutons}) should suffice to model
{scale limits}.
However, the characterisation of these limits is beyond the scope of this paper.
\begin{question}
  Can all scale permutation limits be represented by random permutons?
  If so, which random permutons are scale limits?
\end{question}
We also postpone to future work any consideration of the \emph{packing density} of patterns at a specified scale (see~\cite{PS2010,SS2018}) and, more generally, of the \emph{feasible region} for pattern densities at a specified scale (see~\cite{KKRW2020,BP2020}).

\subsection{Scalable convergence}

We now briefly introduce a stricter notion of convergence in which the choice of scale is immaterial.
We say that an infinite sequence $(\sigma_j)_{j\in\bbN}$ of permutations with $|\sigma_j|\to\infty$ is \emph{scalably convergent} if, for every permutation $\pi$, there exists $\rho_\pi$ such that
$\rho_f(\pi,\sigma_j)$ converges to $\rho_\pi$ for every scaling function $f$.
We call a limit of a scalably convergent sequence a \emph{scalable limit}.
Let us consider some scalable limits which can be represented by permutons.

We say that a permuton is \emph{tiered} if it can be partitioned into a finite or countably infinite number of rectangular horizontal \emph{tiers} $[0,1]\times[a,b]$
such that in each tier the mass is uniformly distributed either on the whole tier or else along the increasing or decreasing diagonal of the tier.
See Figure~\ref{figTieredPermutons} for an illustration of some examples.

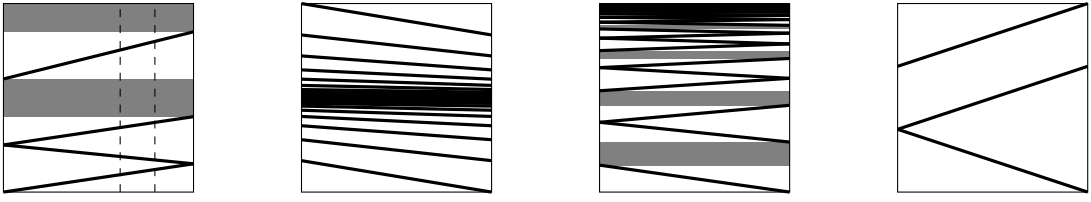
\begin{figure}[t] 
\begin{center}
\begin{tikzpicture}[scale=2.5]
\fill[gray] (0,.4) rectangle (1,.6);
\fill[gray] (0,.85) rectangle (1,1);
\draw[very thick] (0,0)--(1,0.15);
\draw[very thick] (0,0.25)--(1,0.15);
\draw[very thick] (0,0.25)--(1,0.4);
\draw[very thick] (0,0.6)--(1,0.85);
\draw[thin] (0,0) rectangle (1,1);
\draw[thin,dashed] (.6141,0)--(.6141,1);
\draw[thin,dashed] (.7967,0)--(.7967,1);
\end{tikzpicture}
$\qquad\quad$
\begin{tikzpicture}[scale=2.5]
\draw[very thick] (0,.5+1/2)--(1,.5+.666667/2);
\draw[very thick] (0,.5+.666667/2)--(1,.5+.444444/2);
\draw[very thick] (0,.5+.444444/2)--(1,.5+.296296/2);
\draw[very thick] (0,.5+.296296/2)--(1,.5+.197531/2);
\draw[very thick] (0,.5+.197531/2)--(1,.5+.131687/2);
\draw[very thick] (0,.5+.131687/2)--(1,.5+.0877915/2);
\draw[very thick] (0,.5+.0877915/2)--(1,.5+.0585277/2);
\draw[very thick] (0,.5+.0585277/2)--(1,.5+.0390184/2);
\draw[very thick] (0,.5+.0390184/2)--(1,.5+.0260123/2);
\draw[very thick] (0,.5+.0260123/2)--(1,.5+.0173415/2);
\draw[very thick] (0,.5+.0173415/2)--(1,.5+.011561/2);
\draw[very thick] (0,.5+.011561/2)--(1,.5+.00770735/2);
\draw[very thick] (0,.5+.00770735/2)--(1,.5+.00513823/2);
\draw[very thick] (1-0,.5-1/2)--(1-1,.5-.666667/2);
\draw[very thick] (1-0,.5-.666667/2)--(1-1,.5-.444444/2);
\draw[very thick] (1-0,.5-.444444/2)--(1-1,.5-.296296/2);
\draw[very thick] (1-0,.5-.296296/2)--(1-1,.5-.197531/2);
\draw[very thick] (1-0,.5-.197531/2)--(1-1,.5-.131687/2);
\draw[very thick] (1-0,.5-.131687/2)--(1-1,.5-.0877915/2);
\draw[very thick] (1-0,.5-.0877915/2)--(1-1,.5-.0585277/2);
\draw[very thick] (1-0,.5-.0585277/2)--(1-1,.5-.0390184/2);
\draw[very thick] (1-0,.5-.0390184/2)--(1-1,.5-.0260123/2);
\draw[very thick] (1-0,.5-.0260123/2)--(1-1,.5-.0173415/2);
\draw[very thick] (1-0,.5-.0173415/2)--(1-1,.5-.011561/2);
\draw[very thick] (1-0,.5-.011561/2)--(1-1,.5-.00770735/2);
\draw[very thick] (1-0,.5-.00770735/2)--(1-1,.5-.00513823/2);
\draw[thin] (0,0) rectangle (1,1);
\end{tikzpicture}
$\qquad\quad$
\begin{tikzpicture}[scale=2.5]
\fill[gray] (0,0.142857) rectangle (1,0.265306);
\fill[gray] (0,0.460225) rectangle (1,0.537336);
\fill[gray] (0,0.708643) rectangle (1,0.750265);
\fill[gray] (0,0.865199) rectangle (1,0.884457);
\fill[gray] (0,0.946542) rectangle (1,0.954179);
\fill[gray] (0,0.981829) rectangle (1,0.984425);
\draw[very thick] (0,0.142857)--(1,0.);
\draw[very thick] (0,0.370262)--(1,0.265306);
\draw[very thick] (0,0.370262)--(1,0.460225);
\draw[very thick] (0,0.537336)--(1,0.603431);
\draw[very thick] (0,0.660083)--(1,0.603431);
\draw[very thick] (0,0.660083)--(1,0.708643);
\draw[very thick] (0,0.750265)--(1,0.785942);
\draw[very thick] (0,0.816521)--(1,0.785942);
\draw[very thick] (0,0.816521)--(1,0.842733);
\draw[very thick] (0,0.865199)--(1,0.842733);
\draw[very thick] (0,0.900963)--(1,0.884457);
\draw[very thick] (0,0.900963)--(1,0.915111);
\draw[very thick] (0,0.927238)--(1,0.915111);
\draw[very thick] (0,0.927238)--(1,0.937633);
\draw[very thick] (0,0.946542)--(1,0.937633);
\draw[very thick] (0,0.960725)--(1,0.954179);
\draw[very thick] (0,0.960725)--(1,0.966336);
\draw[very thick] (0,0.971145)--(1,0.966336);
\draw[very thick] (0,0.971145)--(1,0.975267);
\draw[very thick] (0,0.9788)--(1,0.975267);
\draw[very thick] (0,0.9788)--(1,0.981829);
\draw[very thick] (0,0.984425)--(1,0.98665);
\draw[very thick] (0,0.988557)--(1,0.98665);
\draw[very thick] (0,0.988557)--(1,0.990192);
\draw[thin] (0,0) rectangle (1,1);
\end{tikzpicture}
$\qquad\quad$
\begin{tikzpicture}[scale=2.5]
\draw[very thick] (0,0.333)--(1,0);
\draw[very thick] (0,0.333)--(1,0.667);
\draw[very thick] (0,0.667)--(1,1);
\draw[thin] (0,0) rectangle (1,1);
\end{tikzpicture}
\caption{Some tiered permutons; the dashed lines delineate a vertical strip}\label{figTieredPermutons}
\end{center}
\end{figure}

Tiered permutons have the property that any vertical strip is equivalent to the whole permuton, in the following sense.
Suppose $\Gamma$ is a tiered permuton and consider a vertical strip $[a,b]\times[0,1]$ of $\Gamma$, as delineated by the dashed lines in the example at the left of Figure~\ref{figTieredPermutons}.
Now let $\Gamma_{[a,b]}$ be the permuton that results from rescaling this strip to fill $[0,1]^2$ in such a way that the result has a total mass of one and uniform marginals, and is thus a valid permuton.
This requires horizontal expansion by a factor of $1/(b-a)$ and the vertical expansion of each line segment so as to become a diagonal of its tier.
Formally, the following defines $\Gamma_{[a,b]}$:
\[
\Gamma_{[a,b]} \big([0,x] \,\times\, [0,y]\big) \;=\;
\tfrac{1}{b-a} \Gamma\big([a,a+x(b-a)] \,\times\, [0,h]\big),
\]
for any $(x,y)\in[0,1]^2$, where $h$ is any solution of the equation
$\Gamma\big([a,b] \times [0,h]\big) = (b-a)y$.
Given any tiered permuton $\Gamma$ and interval $[a,b]\subseteq[0,1]$, it is easy to see that $\Gamma_{[a,b]}=\Gamma$.

In particular, $\Gamma_{[a,a+f(n)/n]}=\Gamma$ for any scaling function $f$, all $a\in[0,1)$, and all $n$ large enough that $a+f(n)/n\leqs1$. 
Thus, for any scaling function $f$ and every pattern $\pi$, we have $\rho_f(\pi,\Gamma)=\rho(\pi,\Gamma)$.
Hence, if $\Gamma$ is tiered, every $\Gamma$-random sequence is scalably convergent to its global limit $\Gamma$.
We believe that tiered permutons are the only permutons with this property:
\begin{conj}\label{conjTiered}
  If $\Gamma$ is a permuton for which every $\Gamma$-random sequence is scalably convergent, then $\Gamma$ is tiered.
\end{conj}
More generally, it seems likely that scalable limits can be characterised as probability distributions over tiered permutons:
\begin{question}
  Can all scalable permutation limits be represented by random tiered permutons?
  If so, which random tiered permutons are scalable limits?
\end{question}

Note that there exist sequences of permutations that converge to a tiered permuton but which are not scalably convergent.
For example, from~\cite{BevanLocallyUniform} we know that if we let $\sigma_j$ be drawn uniformly at random from those permutations of length $j^2$ that have $j^3$ inversions, then, with probability one, $(\sigma_j)_{j\in\bbN}$ converges to the increasing permuton $\!\gcone{1}{1}\!$, but is not scalably convergent;
indeed, at any scale $f\ll n^{1/4}$, the scale limit is the uniform permuton.

Recall that a \mbox{$\Gamma_\textsf{V}$-random} sequence converges to the same limit at any scale $f$, as long as $1\ll f\ll n$.
Thus a \mbox{$\Gamma_\textsf{V}$-random} sequence is scalably convergent, but its scalable limit is not equal to~$\Gamma_\textsf{V}$.
It seems reasonable to believe that the fact that $\Gamma_\textsf{V}$-random sequences are scalably convergent is not due to any specific properties of $\Gamma_\textsf{V}$:
\begin{conj}\label{conjGammaRandomScalablyConv}
  If $\Gamma$ is any permuton, then every $\Gamma$-random sequence is scalably convergent.
\end{conj}

\subsection{Local convergence}

We conclude this section with a very brief foray into local convergence, the theory of which was recently developed by Borga in~\cite{Borga2020}.
An infinite sequence $(\sigma_j)_{j\in\bbN}$ of permutations with $|\sigma_j|\to\infty$ is said to be \emph{locally convergent} if $\rho_{|\pi|}(\pi,\sigma_j)$ converges for every pattern $\pi$.
One can take the \emph{local limit} of a locally convergent sequence of permutations to be a \emph{shift-invariant random infinite rooted permutation} ({\small SIRIRP}).
By~\cite[Proposition~2.44 and Theorem~2.45]{Borga2020} and~\cite[Proposition~3.4]{BP2020}, in an analogous manner to the theory of global limits,
every locally convergent sequence of permutations has a {\small SIRIRP} as a local limit, and every {\small SIRIRP} is the local limit of some locally convergent sequence of permutations.

\section{Scale limits}\label{sectScaleSpecificLimits}

In this section, we prove some fundamental results concerning scale limits.
We begin by establishing (Theorem~\ref{thmLimScaleAny}) that scale limits can be limits at any scale.
We then show (Proposition~\ref{propEveryLength}) that, for any scale limit $\Xi$ and scaling function $f$, there exists a sequence of permutations $(\tau_\ell)_{\ell\in\bbN}$
convergent to $\Xi$ at scale~$f$
with the property that $|\tau_\ell|=\ell$ for each $\ell\in\bbN$.
Finally (Proposition~\ref{propNeedDomination}), we demonstrate that, in general, convergence at scale $f$ is not independent of convergence at scale $cf$, if $c$ is a constant.

We begin by determining the asymptotics of the number of $k$-element subsets of $[n]$ of width no greater than~$f$.
\begin{prop}\label{propNKFAsymptotics}
  If $k\ll f\ll n$, then the number of $k$-element subsets of $[n]$ of width no greater than~$f$ is asymptotic to ${nf^{k-1}}/{(k-1)!}$.
\end{prop}
\begin{proof}
For $f\in\bbN$, 
\begin{align*}
  \binom{n}{k}_{\!f}
  &\;=\; \sum_{w=k}^{f}(n - w + 1)\binom{w-2}{k-2} \\
  &\;=\; \frac{(f+1-k)(nk-fk+f)}{k(k-1)}\binom{f-1}{k-2} \\
  &\;\sim\; \frac{fnk}{k(k-1)}\frac{f^{k-2}}{(k-2)!} \\
  &\;=\; \frac{nf^{k-1}}{(k-1)!},
\end{align*}
where each term in the sum is the number of $k$-element subsets of $[n]$ of width $w$, and the second identity can be established by induction over $f$.
\end{proof}

The remainder of our proofs make use of two standard operations on permutations.
Given two permutations $\sigma$ and $\tau$ with lengths $k$ and $\ell$ respectively, their \emph{direct sum}\label{defDirectSum} $\sigma\oplus\tau$ is the permutation of length $k+\ell$ consisting of $\sigma$ followed by a shifted copy of $\tau$:
\[
(\sigma\oplus\tau)(i) \;=\;
\begin{cases}
  \sigma(i)   & \text{if~} i\leqslant k , \\
  k+\tau(i-k) & \text{if~} k+1 \leqslant i\leqslant k+\ell .
\end{cases}
\]
We also use $\bigoplus^c\sigma$ to denote 
the direct sum of $c$ copies of $\sigma$.
See the left of Figure~\ref{figSubstitution} for an illustration.
Note that permutation inversion (which reflects the plot about the main diagonal) distributes over direct sum: $(\sigma\oplus\tau)^{-1}=\sigma^{-1}\oplus\tau^{-1}$.
\begin{figure}[t] 
\begin{center}
\begin{tikzpicture}[scale=0.18]
\plotpermthinbox{24}{4,2,3,1, 7,9,5,8,6, 12,14,10,13,11, 17,19,15,18,16, 22,24,20,23,21}
\end{tikzpicture}
\qquad\qquad\quad
\begin{tikzpicture}[scale=0.18]
\plotpermthinbox{24}{9,10,13,14,11,12,15,16,1,2,5,6,3,4,7,8,17,18,21,22,19,20,23,24}
\end{tikzpicture}
\caption{Plots of $4231\oplus\bigoplus^4 35142$ (left) and $213[1324][12]$ (right)}\label{figSubstitution}
\end{center}
\end{figure}
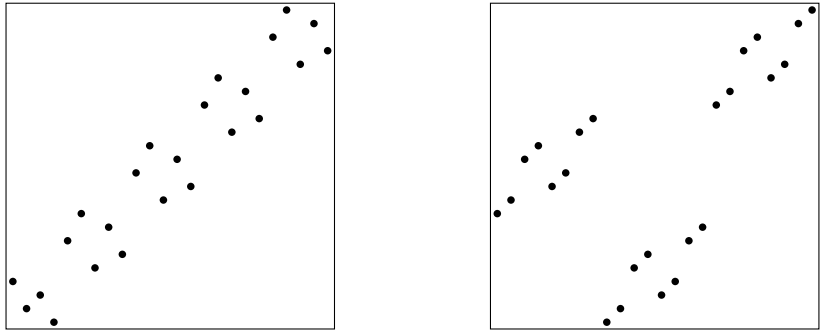

We also make use of \emph{substitution}.
Suppose $\sigma\in S_k$ and $\tau\in S_\ell$, then we denote by $\sigma[\tau]$ the permutation of length $k\+\ell$ created by replacing each point $(i,\sigma(i))$ in the plot of $\sigma$ with a small copy of $\tau$.
Note that substitution is associative: $\sigma[\tau][\upsilon]=\sigma[\tau[\upsilon]]$.
See the right of Figure~\ref{figSubstitution} for an illustration.
Note also that inversion distributes over substitution: $\sigma[\tau]^{-1}=\sigma^{-1}[\tau^{-1}]$.

Most of the subsequent proofs have a very similar structure.
We outline the argument here, so we can abbreviate the proofs below.
Suppose a sequence of permutations $(\sigma_j)_{j\in\bbN}$ converges at scale $f$ to a scale limit~$\Xi$.
From this sequence we construct another sequence $(\tau_m)_{m\in\bbN}$, each term $\tau_m$ being built from multiple copies of a specific $\sigma_{j}$ using directs sums and/or substitution, where $j=j(m)$ increases with $m$.
We desire to prove that $(\tau_m)_{m\in\bbN}$ converges to $\Xi$ at scale~$g$.
It suffices to show that, for any pattern $\pi$, we have $\liminfty[m]\rho_g(\pi,\tau_m)=\liminfty[j]\rho_f(\pi,\sigma_j)$.

Suppose $|\pi|=k$. Then $\rho_g(\pi,\tau_m) = \prob{\tau_m(K) = \pi}$, where $K$ is drawn uniformly from the family $\FFF$ consisting of \mbox{$k$-element} subsets of $[|\tau_m|]$ of width no greater than $g(|\tau_m|)$.
Let us say that $K$ is \emph{good} if each of its members is drawn from a single copy of $\sigma_{j}$ used to build $\tau_m$ and, when restricted to that copy of $\sigma_{j}$, $K$ has width at most $f(|\sigma_{j}|)$.
If $\FFF$ includes every good subset, and the probability that $K$ is \emph{not} good is bounded above by $\Delta_m$, then
\[
\left(\!1-\Delta_m\!\right)\rho_g(\pi,\tau_m)
\;\leqs\;
\rho_f(\pi,\sigma_j)
\;\leqs\;
\left(\!1-\Delta_m\!\right)\rho_g(\pi,\tau_m) + \Delta_m
.
\]
Thus, if $\Delta_m\to0$, the two asymptotic densities of $\pi$ are equal, as required.

The next four propositions establish that scale limits can be limits at any scale.
We begin with a simple technical result for later use.
\begin{prop}\label{propSublinear}
  Suppose a function $f:\bbR^+\to\bbR^+$ satisfies $f(x)\leqs x$ for all $x>0$. If $c>1$ is a constant, then $\liminfinfty[x]f(cx)/f(x)\leqs c$, and thus it is not the case that $f(cx)\gg f(x)$.
\end{prop}
\begin{proof}
  Suppose not, and there exists $d>c$ and $X>0$ such that for all $x\geqs X$ we have $f(cx)/f(x)\geqs d$.
  Then, $f(c^kX)\geqs d^kf(X)$ for each $k\in\bbN$, and so $f(c^kX)>c^kX$ if $k>\log(X/f(X))/\log(d/c)$, in contradiction to the assumption that $f(x)\leqs x$.
\end{proof}
\begin{prop}\label{propLimDown}
  Let $f$ be a scaling function and $f_\bullet$ be another scaling function dominated by~$f$.
  Suppose there exists a sequence of permutations which converges at scale $f$ to a limit~$\Xi$.
  Then there exists a sequence of permutations which converges at scale $f_\bullet$ to~$\Xi$.
\end{prop}
\begin{proof}
  Let $(\sigma_j)_{j\in\bbN}$ be a sequence of permutations which converges at scale $f$ to~$\Xi$, and suppose $|\sigma_j|=u_j$ for each $j\in\bbN$.
  Extend the domains of $f$ and $f_\bullet$ to $\bbR^{\geqs1}$ by interpolation (so both are continuous), and for every $u\geqs1$, let
  \[
  C(u) \;=\; \max\big\{c>0 \::\: f_\bullet(cu)=f(u)\big\} ,
  \]
  which is well-defined because $f\gg f_\bullet\gg1$.
  If~$f_\bullet$ is strictly increasing (and thus has an inverse), then we simply have $C(u)=f_\bullet^{-1}\big(f(u)\big)/u$.
  For example, if $f(n)=n^\alpha$ and $f_\bullet(n)=n^\beta$ with $\alpha>\beta$,
  then 
  $C(u)=u^{\alpha/\beta-1}$.

We claim that $C(u)\to\infty$.
Suppose, to the contrary, that there exists a constant $B>1$ such that $C(u)<B$ for arbitrarily large values of $u$.
Then, for all such $u$, we have $f_\bullet(Bu)>f(u)\gg f_\bullet(u)$.
However, $f_\bullet(u)\leqs u$, so by Proposition~\ref{propSublinear}, 
we do not have $f_\bullet(Bu)\gg f_\bullet(u)$,
a contradiction.

For each $j\in\bbN$, let $c_j=\ceil{C(u_j)}$ and let
$\tau_j$ be the permutation $\bigoplus^{c_j}\sigma_j$ of length $\ell_j=c_ju_j$.
We claim that the sequence $(\tau_j)_{j\in\bbN}$ converges at scale $f_\bullet$ to~$\Xi$.

Let's compare scales:
  \begin{align*}
    f_\bullet(\ell_j)
    &\;=\; f_\bullet\big((1+o(1))C(u_j)u_j\big) \\
    &\;=\; (1+o(1))f_\bullet(C(u_j)u_j), \qquad \text{since $f_\bullet(n)\ll n$,} \\
    &\;=\; (1+o(1))f(u_j) .
  \end{align*}
It suffices to prove that, for any pattern $\pi$, we have $\rho_{f_\bullet}(\pi,\tau_j)\sim\rho_f(\pi,\sigma_j)$.
Suppose $|\pi|=k$.
Then $\rho_{f_\bullet}(\pi,\tau_j) = \prob{\tau_j(K) = \pi}$, where $K$ is drawn uniformly from the \mbox{$k$-element} subsets of $[\ell_j]$ of width no greater than $f_\bullet(\ell_j)$.
Since $\tau_j$ consists of copies of $\sigma_j$, each of length $u_j$,  the probability that $K$ contains points from two copies of $\sigma_j$ is bounded above by $f_\bullet(\ell_j)/u_j\sim f(u_j)/u_j\ll1$.
Moreover, by Proposition~\ref{propNKFAsymptotics}, the probability that $K$ has width at most $f(u_j)$ is asymptotically 
$\big(f_\bullet(\ell_j)/f(u_j)\big)^{k-1}$, which tends to 1 since $k$ is constant.

Thus,
$\liminfty[j]\rho_{f_\bullet}(\pi,\tau_j)=\liminfty[j]\rho_f(\pi,\sigma_j)$. 
\end{proof}

\begin{prop}\label{propLimUp}
  Let $f$ be a scaling function and $f^\bullet$ be another scaling function that dominates~$f$.
  Suppose there exists a sequence of permutations which converges at scale $f$ to a limit~$\Xi$.
  Then there exists a sequence of permutations which converges at scale $f^\bullet$ to~$\Xi$.
\end{prop}
\begin{proof}
  Let $(\sigma_j)_{j\in\bbN}$ be a sequence of permutations which converges at scale $f$ to~$\Xi$, and suppose $|\sigma_j|=u_j$ for each $j\in\bbN$.
  Extend the domains of $f$ and $f^\bullet$ to $\bbR^{\geqs1}$ by interpolation (so both are continuous), and
  for every $u\geqs1$, let
  \[
  D(u) \;=\; \max\big\{d>0 \::\: f^\bullet(du)=d f(u)\big\} .
  \]
  If~$g(n)=f^\bullet(n)/n$ is strictly decreasing, then we simply have $D(u)=g^{-1}\big(f(u)/u\big)/u$.
  For example, if $f(n)=n^\alpha$ and $f^\bullet(n)=n^\beta$ with $\alpha<\beta$,
  then $D(u)=u^{(\beta-\alpha)/(1-\beta)}$.

We claim that $D(u)\to\infty$.
Suppose, to the contrary, that there exists a constant $B>1$ such that $D(u)<B$ for infinitely many values of $u$.
Then, for all such $u$, we have $f(Bu)\gg f^\bullet(Bu)>B\+f(u)$.
But this is impossible, because $f(u)\leqs u$, so by Proposition~\ref{propSublinear}, 
we do not have $f(Bu)\gg f(u)$.

  For each $j\in\bbN$, let $d_j=\ceil{D(u_j)}$ and let
  $\tau_j$ be the permutation $\sigma_j[\vphi_j]$, where $\vphi_j$ is an arbitrary permutation of length $d_j$. Note that $\tau_j$ has length $\ell_j=d_ju_j$.
  We claim that 
  $(\tau_j)_{j\in\bbN}$ converges at scale $f^\bullet$ to~$\Xi$.

  Let's compare scales:
  \begin{align*}
    f^\bullet(\ell_j)
    &\;=\; f^\bullet\big((1+o(1))D(u_j)u_j\big) \\
    &\;=\; (1+o(1))f^\bullet(D(u_j)u_j), \qquad \text{since $f^\bullet(n)\ll n$,} \\
    &\;=\; (1+o(1))D(u_j)f(u_j) \\
    &\;=\; (1+o(1))d_j\+f(u_j) .
  \end{align*}
  It suffices to prove that, for any pattern $\pi$, we have $\rho_{f^\bullet}(\pi,\tau_j)\sim\rho_f(\pi,\sigma_j)$.
Suppose $|\pi|=k$.
Then $\rho_{f^\bullet}(\pi,\tau_j) = \prob{\tau_j(K) = \pi}$, where $K$ is drawn uniformly from the \mbox{$k$-element} subsets of $[\ell_j]$ of width no greater than $f^\bullet(\ell_j)$.
Since $\tau_j$ consists of copies of $\vphi_j$, each of length $d_j$, the probability that $K$ contains two or more points from the same copy of $\vphi_j$ is bounded above by $\binom{k}{2}{d_j}/{f^\bullet(\ell_j)}\sim \binom{k}{2}/f(u_j)\ll1$,
since the probability of any two points being from the same block is less than $1/f(u_j)$.
Moreover, by Proposition~\ref{propNKFAsymptotics}, the probability that $K$ has width at most $d_j\+f(u_j)$ is asymptotically 
$\big(f^\bullet(\ell_j)/d_j\+f(u_j)\big)^{k-1}$, which tends to 1 since $k$ is constant.

Thus,
$\liminfty[j]\rho_{f^\bullet}(\pi,\tau_j)=\liminfty[j]\rho_f(\pi,\sigma_j)$. 
\end{proof}

Our first theorem follows by combining these last two propositions.
\begin{thm}\label{thmLimScaleAny}
  Let $\Xi$ be any scale limit and $f$ be any scaling function.
  Then there exists a sequence of permutations which converges at scale $f$ to~$\Xi$.
\end{thm}
\begin{proof}
  Since $\Xi$ is a scale limit, there is some scaling function $g$ for which there exists a sequence of permutations which converges to $\Xi$ at scale~$g$.
  If either $f\ll g$ or $f\gg g$, the result then follows directly from Proposition~\ref{propLimDown} or Proposition~\ref{propLimUp}, respectively.

  Otherwise, let $h$ be the scaling function defined by $h(n)=\sqrt{\min(g(n),f(n))}$, so both $h\ll g$ and $h\ll f$. Proposition~\ref{propLimDown} can then be applied to give a sequence of permutations which converges at scale $h$ to~$\Xi$. A subsequent application of Proposition~\ref{propLimUp} then yields a sequence of permutations which converges to $\Xi$ at scale~$f$. 
\end{proof}

Our next proposition allows us to assume the existence of a sequence convergent at a specific scale that consists of permutations of every length.
\begin{prop}\label{propEveryLength}
  Suppose $\Xi$ is a scale limit and $f$ a scaling function.
  Then there exists a sequence of permutations $(\tau_\ell)_{\ell\in\bbN}$
  convergent to $\Xi$ at scale~$f$
  with the property that $|\tau_\ell|=\ell$ for each $\ell\in\bbN$.
\end{prop}
\begin{proof}
  Suppose $(\sigma_j)_{j\in\bbN}$ is a sequence of permutations which converges at scale $f$ to~$\Xi$, and that
  $|\sigma_j|=u_j$ for each $j\in\bbN$.
  Such a sequence exists by Theorem~\ref{thmLimScaleAny}.
  By taking a subsequence, we may assume that $u_j$ is strictly increasing.
  We also extend the domain of $f$ to $\bbR^{\geqs1}$ by interpolation.

Let $L:\bbR^+\to\bbR^+$ be a continuous, positive, strictly increasing function. 

  For every $\ell\geqs\ell_1= L(u_1)$, let $j(\ell)=\max\{j\in\bbN:L(u_j)\leqs \ell\}$ select an index into the sequence~$(\sigma_j)_{j\in\bbN}$, and let $u(\ell)=u_{j(\ell)}$ be the length of the corresponding permutation.
  Note that $u(\ell)$ is a weakly increasing integer-valued step function, and that $L(u(\ell))\leqs \ell$, so we have $u(\ell)\leqs L^{-1}(\ell)$.

  For every $\ell\geqs\ell_1$, let $D(\ell)={f(\ell)}/{f(u(\ell))}$.
  We claim that if $L$ grows sufficiently fast, then $D(\ell)\to\infty$.
  Let $f_0:\bbR^+\to\bbR^+$ be a continuous positive strictly increasing function such that $f_0(\ell)\leqs f(\ell)$ for all $\ell$, and suppose
  $L(u)\geqs f_0^{-1}(u^2)$.
  Then,
  \[
  \textstyle
  f(u(\ell)) \;\leqs\; u(\ell) \;\leqs\; L^{-1}(\ell) \;\leqs\; \sqrt{f_0(\ell)} \;\leqs\; \sqrt{f(\ell)}
  ,
  \]
  so $D(\ell)\geqs\sqrt{f(\ell)}\gg1$.

Now, for every $\ell\geqs\ell_1$, let $C(\ell)=\dfrac{\ell}{f(\ell)}\dfrac{f(u(\ell))}{u(\ell)}$,
  so that $C(\ell)\+D(\ell)\+u(\ell)=\ell$.

  If $L$ grows fast enough, then $C(\ell)\to\infty$: Let $g(\ell)=\ell/f(\ell)$. Then $C(\ell)={g(\ell)}/{g(u(\ell))}$.
  So by the argument used to show that $D$ diverges, we can choose $L$ so that $C(\ell)\geqs\sqrt{g(\ell)}$, which tends to infinity because $f(\ell)\ll\ell$.

  For each positive integer $\ell<\ell_1$, let $\tau_\ell$ be any permutation of length $\ell$.
  For each integer $\ell\geqs\ell_1$, let $d_\ell=\floor{D(\ell)}$ and $c_\ell=\floor{C(\ell)}$, and let $\tau_\ell$ be the permutation
  \[
  \psi_\ell \:\oplus\: \bigoplus\nolimits^{c_\ell}\sigma_{j(\ell)}[\vphi_\ell],
  \]
  where
  $\psi_\ell$ is any permutation of length $\ell-c_\ell\+d_\ell\+u(\ell)$, which may be zero,
  and $\vphi_\ell$ is an arbitrary permutation of length $d_\ell$.
  Note that $\tau_\ell$ has length $\ell$, as required,
  since $|\sigma_{j(\ell)}|=u(\ell)$.
  Note also that 
  $|\psi_\ell| < \big((c_\ell+1)(d_\ell+1) - c_\ell\+d_\ell \big)u(\ell) = (c_\ell+d_\ell+1)\+u(\ell) \ll \ell$.
  We claim that 
  $(\tau_\ell)_{\ell\in\bbN}$ converges at scale $f$ to~$\Xi$.

  Let's compare scales:
  $
    f(\ell)
    = D(\ell)\+f(u(\ell))
    = (1+o(1))d_\ell\+f(u(\ell))
  $.

It suffices to prove that, for any pattern $\pi$, we have $\rho_{f}(\pi,\tau_\ell)\sim\rho_f(\pi,\sigma_{j(\ell)})$.
Suppose $|\pi|=k$.
Then $\rho_{f}(\pi,\tau_\ell) = \prob{\tau_\ell(K) = \pi}$, where $K$ is drawn uniformly from the \mbox{$k$-element} subsets of $[\ell]$ of width no greater than $f(\ell)$.
Now the probability that $K$ contains a point from $\psi_\ell$ is bounded above by $|\psi_\ell|/(\ell-f(\ell))\ll1$.
Also, since $\tau_\ell$ consists of copies of $\sigma_{j(\ell)}[\vphi_\ell]$, each of length $d_\ell\+u(\ell)$, the probability that $K$ contains points from two copies of $\sigma_{j(\ell)}[\vphi_\ell]$ is bounded above by $f(\ell)/d_\ell\+u(\ell)\sim f(u(\ell))/u(\ell)\ll1$.
Moreover, since $\tau_\ell$ consists of copies of $\vphi_\ell$, each of length $d_\ell$, the probability that $K$ contains two or more points from the same copy of $\vphi_\ell$ is bounded above by $\binom{k}{2}{d_\ell}/{f(\ell)}\sim \binom{k}{2}/f(u(\ell))\ll1$, since $k$ is constant.
Furthermore, by Proposition~\ref{propNKFAsymptotics}, the probability that $K$ has width at most $d_\ell\+f(u(\ell))$ is asymptotically 
$\big(f(\ell)/d_\ell\+f(u(\ell))\big)^{k-1}$, which tends to 1 since $k$ is constant.

Thus,
$\liminfty[\ell]\rho_{f}(\pi,\tau_\ell)=\liminfty[j]\rho_f(\pi,\sigma_j)$.
\end{proof}

In a similar manner, we may also assume that the existence of a locally convergent sequence convergent that consists of permutations of every length.
We provide a proof, since this doesn't seem to be in the literature.
\begin{prop}\label{propEveryLengthLocal}
  Suppose $\Lambda$ is a permutation local limit.
  Then there exists a sequence of permutations $(\tau_\ell)_{\ell\in\bbN}$
  that converges locally to $\Lambda$
  with the property that $|\tau_\ell|=\ell$ for each $\ell\in\bbN$.
\end{prop}
\begin{proof}
  Suppose $(\sigma_j)_{j\in\bbN}$ is a sequence of permutations which converges locally to~$\Lambda$, and that
  $|\sigma_j|=u_j$ for each $j\in\bbN$. By taking a subsequence, we may assume that $u_j$ is strictly increasing.

For every $\ell\geqs\ell_1= {u_1}^{\!2}$, let $j(\ell)=\max\{j\in\bbN:{u_j}^{\!2}\leqs \ell\}$ select an index into the sequence~$(\sigma_j)_{j\in\bbN}$, and let $u(\ell)=u_{j(\ell)}$ be the length of the corresponding permutation.
Note that 
$u(\ell)^2\leqs \ell$, so $u(\ell)\leqs\sqrt{\ell}$.
Also, for every $\ell\geqs\ell_1$, let $C(\ell)=\ell/u(\ell)$, so that $C(\ell)\+u(\ell)=\ell$.
Note that $C(\ell)\geqs\sqrt{\ell}$, so $C(\ell)\to\infty$.

  For each positive integer $\ell<\ell_1$, let $\tau_\ell$ be any permutation of length $\ell$.
  For each integer $\ell\geqs\ell_1$, let $c_\ell=\floor{C(\ell)}$, and let $\tau_\ell$ be the permutation
  \[
  \psi_\ell \:\oplus\: \bigoplus\nolimits^{c_\ell}\sigma_{j(\ell)},
  \]
  where
  $\psi_\ell$ is any permutation of length $\ell-c_\ell\+u(\ell)$, which may be zero.
  Note that $\tau_\ell$ has length~$\ell$, as required. 
  Note also that $|\psi_\ell| < u(\ell) \ll \ell$.
  We claim that 
  $(\tau_\ell)_{\ell\in\bbN}$ converges locally to~$\Lambda$.

It suffices to prove that, for any pattern $\pi$, we have $\rho_{|\pi|}(\pi,\tau_\ell)\sim\rho_{|\pi|}(\pi,\sigma_{j(\ell)})$.
  Suppose $|\pi|=k$.
  Then $\rho_k(\pi,\tau_\ell) = \prob{\tau_\ell(K) = \pi}$, where $K$ is drawn uniformly from the subintervals of $[\ell]$ of width $k$.
Now the probability that $K$ contains a point from $\psi_\ell$ is bounded above by $|\psi_\ell|/(\ell-k)\ll1$.
Also, since $\tau_\ell$ consists of copies of $\sigma_{j(\ell)}$, each of length $u(\ell)$, the probability that $K$ contains points from two copies of $\sigma_{j(\ell)}$ is bounded above by $k/u(\ell)\ll1$.

  Thus,
  $\liminfty[\ell]\rho_{k}(\pi,\tau_\ell)=\liminfty[j]\rho_k(\pi,\sigma_j)$. 
\end{proof}

In the next section, we establish our primary result, showing asymptotic independence at a countably infinite number of scales.
As a condition for this, we require that the set of scales is \emph{totally ordered by domination}, that is, for every distinct pair of scaling functions $f$ and $g$, either $f\ll g$ or $g\ll f$.\footnote{For a delightful elementary exposition of the properties of the poset of real-valued functions under ordering by domination, see Hardy's \emph{Orders of infinity}~\cite{Hardy1910}.}
We conclude this section by proving a result that shows that we need this condition, or one like it: limits at scales whose ratio is a constant are not in general independent.
\begin{prop}\label{propNeedDomination}
If a sequence of permutations converges at scale $f$ to the increasing permuton $\!\gcone{1}{1}\!$,
then for any constant $c\in(0,1)$, at scale $cf$ the sequence also converges to~$\!\gcone{1}{1}\!$.
\end{prop}
\begin{proof}
  Suppose $(\sigma_j)_{j\in\bbN}$ converges at scale $f$ to the increasing permuton; that is,
  \[
  \liminfty[j]\rho_f(\pi,\sigma_j)
  \;=\;
  \begin{cases}
    1, & \text{if $\pi$ is an increasing permutation,} \\
    0, & \text{otherwise.}
  \end{cases}
  \]
Suppose further, and contrary to the claim, that for some positive constant $c<1$, some non-increasing $\pi\in S_k$ and some $\rho>0$, we have
  $
  \limsup\rho_{cf}(\pi,\sigma_j) = \rho
  $.
  Then $(\sigma_j)_{j\in\bbN}$ has a subsequence, $(\tau_j)_{j\in\bbN}$ say, such that $\rho_{cf}(\pi,\tau_j)\to \rho$.
  Suppose that $|\tau_j|=n_j$ for each $j\in\bbN$, and let $\iota_k=1\ldots k$ denote the increasing permutation of length $k$.

  Given that $cf<f$, 
  \[
  \nu_f(\iota_k,\tau_j) \:+\: \nu_{cf}(\pi,\tau_j) \;\leqs\; \binom{n_j}{k}_{\!f},
  \]
  since the left-hand side is the number of occurrences of certain length-$k$ patterns of width no greater than $f$ and the right-hand side is the total number of
  occurrences of \emph{all} length-$k$ patterns of width at most $f$.

  However, $\nu_f(\iota_k,\tau_j) \sim \binom{n_j}{k}_f$ since $(\tau_j)_{j\in\bbN}$ converges to~$\!\gcone{1}{1}\!$ at scale~$f$.
  Moreover, by Proposition~\ref{propNKFAsymptotics},
  \[
  \nu_{cf}(\pi,\tau_j)
  \;\sim\;
  \rho \frac{n_j\+(cf)^{k-1}}{(k-1)!}
  \;=\;
  \rho \+ c^{k-1} \frac{n_j\+f^{k-1}}{(k-1)!}
  \;\sim\;
  \rho \+ c^{k-1} \binom{n_j}{k}_{\!f}
  ,
  \]
  and so $\nu_f(\iota_k,\tau_j) + \nu_{cf}(\pi,\tau_j) \sim (1+\rho\+c^{k-1}) \binom{n_j}{k}_f$, which is a contradiction because $\rho\+c^{k-1}$ is a positive constant.
\end{proof}

\section{Independence of limits}\label{sectIndependence}

In this final section, we first prove
independence of limits at countably infinite scales (Theorem~\ref{thmIndependence}).
Then, after briefly presenting two examples, we extend this result to independence of limits in two directions (Theorem~\ref{thm2DIndependence}).

\subsection{Independence of scale limits}

\begin{thm}\label{thmIndependence}
Let $\{f_t : t\in\bbN \}$ be any countably infinite set of scaling functions totally ordered by domination. 
For each $t\in\bbN$, let $\Xi_t$ be any scale limit.
Let $\Gamma$ be any permuton and $\Lambda$ be any permutation local limit.
Then there exists a sequence of permutations which
converges to~$\Xi_t$ at scale $f_t$ for each $t\in\bbN$,
converges globally to $\Gamma$,
and converges locally to~$\Lambda$.
\end{thm}
\begin{proof}
We build a suitable sequence $(\tau_m)_{m\in\bbN}$ by iteratively combining the following \emph{component sequences} using substitution:
\begin{bullets}
  \item Let $(\sigma_0^j)_{j\in\bbN}$ be a convergent sequence of permutations with global limit $\Gamma$, such that $|\sigma_0^j|=j$ for each $j\in\bbN$.
  A $\Gamma$-random sequence would suffice.
  \item For each $t\in\bbN$,
  given a scaling function $e_t$ to be specified below,
  let $(\sigma_t^j)_{j\in\bbN}$ be a sequence of permutations that converges to~$\Xi_t$ at scale~$e_t$, such that $|\sigma_t^j|=j$ for each $j\in\bbN$.
  Proposition~\ref{propEveryLength} guarantees that such a sequence exists.
  \item Let $(\sigma_\infty^j)_{j\in\bbN}$ be a locally convergent sequence of permutations with local limit $\Lambda$, such that $|\sigma_\infty^j|=j$ for each $j\in\bbN$.
  The existence of such a sequence is guaranteed by Proposition~\ref{propEveryLengthLocal}.
\end{bullets}
Fix an index $m\geqs1$, and let $f^1\gg f^2\gg \ldots \gg f^m$ (with superscript indices) be the total ordering by domination of the $m$ scaling functions $f_1,\ldots,f_m$.
So, for each ``level'' $\ell\in[m]$, there is a distinct $t(\ell,m)\in[m]$ such that $f^\ell=f_{t(\ell,m)}$.
Similarly, for each $t\leqs m$, there is a distinct level $\ell(t,m)\in[m]$ such that $f^{\ell(t,m)}=f_{t}$.

Note that we cannot assume that the scaling functions are initially ordered by domination $f_1\gg f_2\gg\ldots$, since that would
restrict the applicability of the theorem to sets of scaling functions that are {well-ordered by domination with order type $\omega$}. Neither example in Section~\ref{sectExamples} below satisfies this restriction, both having a set of scaling functions of the form $\{n^\alpha:\alpha\in A\}$ for a set $A$ dense in $(0,1)$.

Let $f^0=n$ and $f^{m+1}=1$, and for each $\ell\in[m+1]$, let $h_\ell=\sqrt{f^{\ell-1}\+f^{\ell}}$ be a scaling function that is dominated by $f^{\ell-1}$ and dominates~$f^{\ell}$.
Let $h_0=n$ and $h_{m+2}=1$.
Thus, we have the following ordering:
\[
n \:=\: h_0 \:=\: f^0 \:\gg\: h_1 \:\gg\: f^1 \:\gg\: h_2 \:\gg\: f^2 \:\gg\: \:\ldots\: \:\gg\: f^m \:\gg\: h_{m+1} \:\gg\: f^{m+1} \:=\: h_{m+2} \:=\: 1  .
\]
For each $\ell\in[0,m+1]$, let $g_\ell=h_\ell/h_{\ell+1}$ be the ``size of the gap'' between $h_\ell$ and $h_{\ell+1}$.
Note that $h_\ell=\prod_{r=\ell}^{m+1} g_r$.

We want each gap to grow sufficiently fast with $m$. Let
\[
N_m \;=\; \min \left\{N\in\bbN \::\: \frac{f^{\ell-1}(n)}{h_\ell(n)}\geqs m\: \text{~and~} \:\frac{h_{\ell}(n)}{f^\ell(n)}\geqs m\: \text{~for all $n\geqs N$ and each $\ell\in[m+1]$}\right\}.
\]
This definition is valid because $f^{\ell-1}\gg h_\ell\gg f^\ell$ for each $\ell$.

Now, for each $\ell\in[0,m+1]$, let $M^\ell_m=\ceil{g_\ell(N_m)}$.
For $\ell\in[m]$, this is always at least $m^2$, while $M^0_m$ and $M^{m+1}_m$ are no less than $m$.
These are the sizes of the terms from the component sequences
to be used in the construction.
Note that for each $\ell$, we have $h_\ell(N_m)\sim\prod_{r=\ell}^{m+1} M^r_m$ since their ratio is less than $\big(\frac{m+1}{m}\big)^{\!2}\big(\frac{m^2+1}{m^2}\big)^{\!m}\sim 1$.

For each $\ell\in[m]$, let $\lambda^m_\ell=\sigma_{t(\ell,m)}^{M^\ell_m}$,
and let $\lambda^m_{0}=\sigma_0^{M^{0}_m}$ and $\lambda^m_{m+1}=\sigma_\infty^{M^{m+1}_m}$.
Thus, for each $t$, the sequence $\big(\lambda^m_{\ell(t,m)}\big)_{\!m\geqs t}$
is a subsequence of the component sequence $(\sigma_{t}^j)_{j\in\bbN}$, and
$\big(\lambda^m_0\big)_{\!m\in\bbN}$ and $\big(\lambda^m_{m+1}\big)_{\!m\in\bbN}$
are subsequences of $(\sigma_{0}^j)_{j\in\bbN}$ and $(\sigma_{\infty}^j)_{j\in\bbN}$, respectively.

We now assemble $\tau_m$ by iterated substitution as follows:
\[
\tau_m \;=\; \lambda^m_0 \big[ \lambda^m_1 \big] \big[ \lambda^m_2 \big] \ldots
\big[ \lambda^m_m \big] \big[ \lambda^m_{m+1} \big] .
\]
Let $n_m=\prod_{\ell=0}^{m+1} M^\ell_m$ be the length of $\tau_m$. Note that $n_m\geqs N_m$, but $n_m\sim N_m$.

Also, let $\kappa^m_\ell = \lambda^m_0 \big[ \lambda^m_1 \big] \ldots \big[ \lambda^m_{\ell-1} \big]$
and
$\mu^m_\ell = \lambda^m_{\ell+1} \big[ \lambda^m_{\ell+2} \big] \ldots \big[ \lambda^m_{m+1} \big]$,
so we have the following tripartite decomposition for any $\ell\in[m]$:
\[
\tau_m
\;=\;
\kappa^m_\ell\big[\lambda^m_\ell\big]\big[\mu^m_\ell\big]
\;=\;
\underbrace{\lambda^m_0
\big[ \lambda^m_1 \big]
\ldots
\big[ \lambda^m_{\ell-1} \big]}_{\kappa^m_\ell}
\big[ \lambda^m_\ell \big]
\underbrace{\big[ \lambda^m_{\ell+1} \big]
\ldots
\big[ \lambda^m_{m+1} \big]}_{\mu^m_\ell}
.
\]

\textbf{Local convergence.}
Since $\lambda^m_{m+1}=\sigma_\infty^{M^{m+1}_m}$,
to prove that $(\tau_m)_{m\in\bbN}$ converges locally to $\Lambda$,
it suffices to show, for all~$k$ and every permutation~$\pi$ of length $k$, that $\rho_k(\pi,\tau_m)\sim\rho_k(\pi,\lambda^m_{m+1})$.
Now, $\rho_k(\pi,\tau_m) = \prob{\tau_m(K) = \pi}$, where $K$ is drawn uniformly from the subintervals of $[n_m]$ of width $k$.
Since $\tau_m$
is constructed from copies of $\lambda^m_{m+1}$, each of width $M^{m+1}_m\geqs m$,
the probability that $K$ contains points from two copies of $\lambda^m_{m+1}$ is bounded above by $k/m$, which tends to zero since $k$ is a constant.
Thus,
$\liminfty[m]\rho_k(\pi,\tau_m) = \liminfty[m]\rho_k(\pi,\lambda^m_{m+1})$. 

\textbf{Global convergence.}
To prove that $(\tau_m)_{m\in\bbN}$ converges to $\Gamma$,
it suffices to show that $\rho(\pi,\tau_m)\sim\rho(\pi,\lambda^m_0)$ for every permutation~$\pi$.
Suppose $|\pi|=k$.
Then $\rho(\pi,\tau_m) = \prob{\tau_m(K) = \pi}$, where $K$ is drawn uniformly from the \mbox{$k$-element} subsets of $[n_m]$.
Since $\tau_m$
is constructed from $M^0_m\geqs m$ copies of $\mu^m_0$, the probability that $K$ contains two or more points from the same copy of $\mu^m_0$ is bounded above by $\binom{k}{2}/{m}$. 
Thus,
$\liminfty[m]\rho(\pi,\tau_m) = \liminfty[m]\rho(\pi,\lambda^m_0)$. 

\textbf{Convergence at other scales.}
Fix $t\geqs1$ and let $\ell_m=\ell(t,m)$, so $f^{\ell_m}=f_t$.
Now
$\big(\lambda^m_{\ell_m}\big)_{\!m\geqs t}$
consists of terms from $(\sigma_{t}^j)_{j\in\bbN}$. So
to prove that $(\tau_m)_{m\in\bbN}$ converges to
$\Xi_t$
at scale $f_t$,
it suffices to show that we can choose a scaling function $e_t$
in such a way that $\rho_{f_t}(\pi,\tau_m)\sim\rho_{e_t}\big(\pi,\lambda^m_{\ell_m}\big)$
for every permutation~$\pi$.
Suppose $|\pi|=k$.
Then $\rho_{f_t}(\pi,\tau_m) = \prob{\tau_m(K) = \pi}$, where $K$ is drawn uniformly from the \mbox{$k$-element} subsets of $[n_m]$ of width no greater than $f_t(n_m)$.

Recall that $\tau_m=\kappa^m_{\ell_m}\big[\lambda^m_{\ell_m}\big]\big[\mu^m_{\ell_m}\big]$.
Thus $\tau_m$
\begin{bullets}
\item is constructed from copies of $\lambda^m_{\ell_m}\big[\mu^m_{\ell_m}\big]$,
of length $u_m = \prod_{\ell=\ell_m}^m M^\ell_m\sim h_{\ell_m}(n_m)$,
\item each of which is formed of copies of $\mu^m_{\ell_m}$,
of length $w_m = \prod_{\ell=\ell_m+1}^m M^\ell_m\sim h_{\ell_m+1}(n_m)$.
\end{bullets}
So the probability that $K$ contains points from two copies of $\lambda^m_{\ell_m}\big[\mu^m_{\ell_m}\big]$ is bounded above by $f_t(n_m)/u_m$, which tends to zero because $u_m\sim h_{\ell_m}(n_m)$ and, from the definition of $N_m$, we know that $h_{\ell_m}(n_m)/f_t(n_m)\geqs m$.
Moreover,
the probability that $K$ contains two or more points from the same copy of $\mu^m_{\ell_m}$ is bounded above by $\binom{k}{2}\frac{w_m}{f_t(n_m)}$,
which also tends to zero since $w_m\sim h_{\ell_m+1}(n_m)$
and, from the definition of $N_m$, we know that $f_t(n_m)/h_{\ell_m+1}(n_m)\geqs m$.

To conclude, we define $e_t$ so that $e_t(u_m)=f_t(n_m)$.
This defines a valid scaling function since $f_t(n_m)\ll h_{\ell_m}(n_m)\sim u_m$ and thus $e_t\ll n$, and
yields
$\liminfty[m]\rho_{f_t}(\pi,\tau_m)=\liminfty[m]\rho_{e_t}(\pi,\lambda^m_{\ell_m})$. 
\end{proof}

\subsection{Examples}\label{sectExamples}

\begin{exampleO}
By appropriately choosing each 
$\Xi_t$ in the statement of Theorem~\ref{thmIndependence}
to be 
either the increasing or the decreasing permuton, we can construct
a sequence of permutations $(\zeta_j)_{j\in\bbN}$ such that, for each irreducible $p/q\in\bbQ\+\cap\+(0,1]$, we have the following:
\begin{bullets}
  \item If $q$ is odd, then $(\zeta_j)$ converges at scale $n^{p/q}$ to the increasing permuton \gcone{1}{1}$\!$.
  \item If $q$ is even, then $(\zeta_j)$ converges at scale $n^{p/q}$ to the decreasing permuton \gcone{1}{-1}$\!$.
\end{bullets}
Thus, a length $k$ subpermutation of $\zeta_j$ of width at most $|\zeta_j|^{p/q}$ chosen uniformly at random is
asymptotically almost surely the increasing permutation $12\ldots k$ if $q$ is odd, and is
asymptotically almost surely the decreasing permutation $k\ldots21$ if $q$ is even.
\end{exampleO}

\begin{exampleO}
A \emph{skinny monotone grid class} is a set of permutations defined by a $\pm1$ vector.
Given such a vector $\mathbf{v}=(v_1,\ldots,v_d)$, let $\Gamma_\mathbf{v}$ be the tiered permuton with $d$ tiers of equal height numbered from top to bottom, such that in tier $i$ the mass is on the increasing diagonal if $v_i=1$ and on the decreasing diagonal if $v_i=-1$.
The permuton $\Gamma_{(1,1,-1)}$ is shown at the right of Figure~\ref{figTieredPermutons}.

Given a tiered permuton $\Gamma_\mathbf{v}$, the skinny monotone grid class $\Grid(\mathbf{v})$ contains every permutation that can be sampled from $\Gamma_\mathbf{v}$ as described in Section~\ref{sectGlobalConv} (that is, it consists of every possible \mbox{$\Gamma_\mathbf{v}$-random} permutation).
For more on skinny grid classes, see~\cite[Chapter 3]{BevanThesis} and~\cite{BS2019}.
We can assemble a sequence of permutations $(\eta_j)_{j\in\bbN}$ so that, for every skinny monotone grid class $\Grid(\mathbf{v})$, there is a scale $f_\mathbf{v}$ such that $(\eta_j)$ converges at scale $f_\mathbf{v}$ to $\Gamma_\mathbf{v}$.

First, we associate to each skinny monotone grid class $\Grid(\mathbf{v})$ a unique value $\alpha(\mathbf{v})\in(0,1)$, defined by $\alpha(\mathbf{v}) = \half\big(1+\sum_{i=1}^d {v}_i/2^i\big)$. For example, $\alpha(1,1,-1)=\frac{13}{16}$.
Note that $\alpha$ takes values dense in $(0,1)$.
Let $f_\mathbf{v}(n)=n^{\alpha(\mathbf{v})}$.
Then, by choosing each
$\Xi_t$ in the statement of Theorem~\ref{thmIndependence}
to be the tiered permuton of a distinct skinny monotone grid class,
we can construct $(\eta_j)$ so that,
for each $\pm1$ vector $\mathbf{v}$,
a length~$k$ subpermutation of $\eta_j$ of width at most $|\eta_j|^{\alpha(\mathbf{v})}$
chosen uniformly at random
is asymptotically almost surely a permutation in $\Grid(\mathbf{v})$.
\end{exampleO}

\subsection{Independence in two directions}

Our final goal is to prove that we can choose limits at a countably infinite number of scales independently both for a sequence $(\sigma_j)_{j\in\bbN}$ and for its \emph{inverse sequence} $(\sigma_j^{-1})_{j\in\bbN}$.

To do this, we use the following operation on permutations.
Given permutations $\sigma\in S_k$ and $\tau\in S_\ell$, we define their \emph{box product} $\sigma\boxdot\tau$ to be the permutation of length $k\ell$ satisfying, for each index $i$,
\[
\big(\sigma\boxdot\tau\big)(i)
\;=\;
\ell\Big(\sigma\big((i-1\bmod k)+1\big)-1\Big) \:+\: \tau\big(\!\floor{(i-1)/k}+1\big) .\footnote{If the convention were for indices and values of a $k$-permutation to range from $0$ to $k-1$, rather than from $1$ to $k$, then we would simply have $(\sigma\boxdot\tau)(i)=\ell\+\sigma(i\bmod k) + \tau(\floor{i/k})$, which is rather more insightful.}
\]
\begin{figure}[t] 
\begin{center}
\begin{tikzpicture}[scale=0.3]
\draw[thin,dashed] (3.5,0.7)--(3.5,12.3);
\draw[thin,dashed] (6.5,0.7)--(6.5,12.3);
\draw[thin,dashed] (9.5,0.7)--(9.5,12.3);
\draw[thin,dotted] (0.7,4.5)--(12.3,4.5);
\draw[thin,dotted] (0.7,8.5)--(12.3,8.5);
\plotpermthinbox{12}{8, 4, 12, 7, 3, 11, 5, 1, 9, 6, 2, 10}
\end{tikzpicture}
\caption{The plot of the box product $213\boxdot4312$: each vertical strip is a copy of $213$ and each horizontal strip is copy of $4312$}\label{figBoxProduct}
\end{center}
\end{figure}
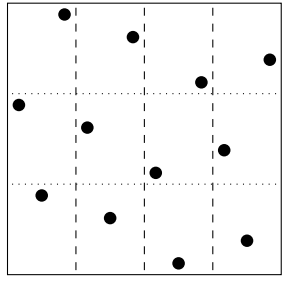
The box product is designed so that it consists of the juxtaposition of $\ell$ copies of $\sigma$, and so that its inverse $(\sigma\boxdot\tau)^{-1}$ consists of $k$ copies of $\tau^{-1}$.
See Figure~\ref{figBoxProduct} for an illustration.

We first establish that limits at a given scale can be chosen independently for a sequence and its inverse sequence.
\begin{prop}\label{prop2D}
  Let $\Xi$ and $\Xi'$ be any 
  scale limits, and $f$ be any scaling function.
  Then there exists a sequence of permutations $(\vphi_j)_{j\in\bbN}$ convergent to $\Xi$ at scale~$f$ such that $(\vphi_j^{-1})_{j\in\bbN}$ converges to $\Xi'$ at scale $f$. 
\end{prop}
\begin{proof}
  Let $g$ be the scaling function defined by $g(n)=f(n^2)$, and suppose $(\sigma_j)_{j\in\bbN}$ and $(\tau_j)_{j\in\bbN}$ converge at scale $g$ to $\Xi$ and $\Xi'$, respectively, with $|\sigma_j|=|\tau_j|=j$ for each $j$.
  Proposition~\ref{propEveryLength} guarantees the existence of these sequences.

  Consider the sequence of permutations $(\vphi_j)_{j\in\bbN}$ where $\vphi_j=\sigma_j\boxdot\tau_j^{-1}$ for each $j$.
  Note that $\vphi_j$ has length $j^2$ and consists of $j$ copies of $\sigma_j$, while $\vphi_j^{-1}$ consists of $j$ copies of $\tau_j$.
  We claim that $(\vphi_j)_{j\in\bbN}$ converges at scale $f$ to~$\Xi$.

It suffices to prove that, for any pattern $\pi$, we have $\rho_{f}(\pi,\vphi_j)\sim\rho_g(\pi,\sigma_j)$.
Suppose $|\pi|=k$.
Then $\rho_{f}(\pi,\vphi_j) = \prob{\vphi_j(K) = \pi}$, where $K$ is drawn uniformly from the \mbox{$k$-element} subsets of $[j^2]$ of width no greater than $f(j^2)=g(j)$.
Since $\vphi_j$ consists of copies of $\sigma_j$, each of length $j$,  the probability that $K$ contains points from two copies of $\sigma_j$ is bounded above by $f(j^2)/j=g(j)/j\ll1$.
Thus,
$\liminfty[j]\rho_{f}(\pi,\vphi_j)=\liminfty[j]\rho_g(\pi,\sigma_j)$.

By a symmetrical argument, we also have $\liminfty[j]\rho_{f}(\pi,\vphi_j^{-1})=\liminfty[j]\rho_g(\pi,\tau_j)$.
\end{proof}

The analogous result holds for local limits.
\begin{prop}\label{prop2DLocal}
  Let $\Lambda$ and $\Lambda'$ be any 
  permutation local limits.
  Then there exists a sequence of permutations $(\vphi_j)_{j\in\bbN}$ that converges locally to $\Lambda$ such that $(\vphi_j^{-1})_{j\in\bbN}$ converges locally to~$\Lambda'$.
\end{prop}
\begin{proof}
Suppose $(\sigma_j)_{j\in\bbN}$ and $(\tau_j)_{j\in\bbN}$ converge locally to $\Lambda$ and $\Lambda'$, respectively.
For each $j\in\bbN$, let $\vphi_j=\sigma_j\boxdot\tau_j^{-1}$ for each $j$.
Note that $\vphi_j$ has length $|\sigma_j||\tau_j|$ and consists of $|\tau_j|$ copies of $\sigma_j$, while $\vphi_j^{-1}$ consists of $|\sigma_j|$ copies of $\tau_j$.
We claim that $(\vphi_j)_{j\in\bbN}$ converges locally to~$\Lambda$.

It suffices to prove that, for any pattern $\pi$, we have $\rho_{|\pi|}(\pi,\vphi_j)\sim\rho_{|\pi|}(\pi,\sigma_j)$.
Suppose $|\pi|=k$.
Then $\rho_k(\pi,\vphi_j) = \prob{\vphi_j(K) = \pi}$, where $K$ is drawn uniformly from the subintervals of $[|\vphi_j|]$ of width $k$.
Now, the probability that $K$ contains points from two copies of $\sigma_j$ is bounded above by $k/|\sigma_j|\ll1$, since $k$ is constant and $|\sigma_j|\to\infty$.
Thus,
$\liminfty[j]\rho_k(\pi,\vphi_j)=\liminfty[j]\rho_k(\pi,\sigma_j)$.

By a symmetrical argument, we also have $\liminfty[j]\rho_k(\pi,\vphi_j^{-1})=\liminfty[j]\rho_k(\pi,\tau_j)$.
\end{proof}
Note that no similar result for global convergence is possible. If $(\sigma_j)_{j\in\bbN}$ converges globally to a permuton $\Gamma$, then $(\sigma_j^{-1})_{j\in\bbN}$ converges globally to $\Gamma^{-1}$, the reflection of $\Gamma$ about the main diagonal.

Recall that permutation inversion distributes over both direct sum and substitution. 
As a result, if a sequence $(\tau_m)_{m\in\bbN}$ is constructed from component sequences by these operations, 
then any relationship between a limit of $(\tau_m)_{m\in\bbN}$ and a limit of one of the component sequences $(\sigma_j)_{j\in\bbN}$ also holds between
the corresponding limits of $(\tau_m^{-1})_{m\in\bbN}$ and $(\sigma_j^{-1})_{j\in\bbN}$.
This observation means that each result above can trivially be adapted to incorporate this two-dimensional perspective and enables us to extend our main result to limits in two directions.
\begin{thm}\label{thm2DIndependence}
Let $\{f_t : t\in\bbN \}$ be any countably infinite set of scaling functions totally ordered by domination.
For each $t\in\bbN$, let $\Xi_t$ and $\Xi'_t$ be any scale limits.
Let $\Gamma$ be any permuton, and let $\Lambda$ and $\Lambda'$ be any permutation local limits.
Then there exists a sequence of permutations $(\tau_j)_{j\in\bbN}$ which
converges to~$\Xi_t$ at scale $f_t$ for every $t$,
converges globally to $\Gamma$,
and converges locally to~$\Lambda$, such that
$(\tau_j^{-1})_{j\in\bbN}$ converges to~$\Xi'_t$ at scale $f_t$ for every $t$ and converges locally to~$\Lambda'$.
\end{thm}
\begin{proof}
  We apply the construction from the proof of Theorem~\ref{thmIndependence} to
\begin{bullets}
  \item a sequence of permutations with global limit~$\Gamma$,
  \item for each $t\in\bbN$, a sequence of permutations that converges to~$\Xi_t$ at scale~$e_t$ whose inverse sequence converges at scale~$e_t$ to~$\Xi'$, and
  \item a sequence of permutations with local limit $\Lambda$ whose inverse sequence converges locally to~$\Lambda'$.
\end{bullets}
   Propositions~\ref{prop2D} and~\ref{prop2DLocal}, along with two-dimensional extensions of
   Propositions~\ref{propEveryLength} and~\ref{propEveryLengthLocal} guarantee the existence of suitable sequences.
   The result follows \emph{mutatis mutandis}.
\end{proof}

\subsection*{Acknowledgements}
The author is grateful for the detailed feedback from three anonymous referees, which resulted in significant improvements to the paper.

\emph{Soli Deo gloria!}

\bibliographystyle{plain}
{\small\bibliography{../bib/mybib}}

\end{document}